\documentclass{amsart}

\usepackage{amsmath,amscd,amsbsy,amsfonts,microtype}
\usepackage{latexsym,amssymb,mathrsfs,bbm,enumerate,verbatim}
\usepackage[all]{xy}
\usepackage{graphicx}
\usepackage[pagebackref,colorlinks,citecolor=blue,linkcolor=blue,urlcolor=blue,filecolor=blue]{hyperref}
\usepackage{tikz}
\usetikzlibrary{calc}

\newcounter{prob}
\newcommand{\Q}{\mathbb{Q}}

\newcommand{\Z}{\mathbb{Z}}
\newcommand{\C}{\mathbb{C}}

\newcommand{\m}{\longrightarrow}

\newtheorem{theorem}{Theorem}[section]
\newtheorem{lemma}[theorem]{Lemma}
\newtheorem{proposition}[theorem]{Proposition}

\newtheorem{remark}[theorem]{Remark}

\title[Stable homology and Occam's razor]{Some stable homology calculations and Occam's razor for Hodge structures}
\author{Alexander Kupers}
\thanks{Alexander Kupers is supported by a William R. Hewlett Stanford Graduate Fellowship, Department of Mathematics, Stanford University, and was partially supported by NSF grant DMS-1105058.}
\author{Jeremy Miller}

\date{\today}

\begin{document}

\begin{abstract}
Motivated by motivic zeta function calculations, Vakil and Wood in \cite{VW} made several conjectures regarding the topology of subspaces of symmetric products. The purpose of this note is to prove two of these conjectures and disprove a strengthening of one of them.  
  \end{abstract}

\maketitle


\section{Introduction} 

In \cite{VW} Vakil and Wood introduced a method for predicting the rational homology groups of subspaces of symmetric products of complex varieties and used it to formulate several conjectures regarding these homology groups. This method uses motivic zeta functions and a principle they dub ``Occam's razor for Hodge structures.'' Roughly, the idea of this principle is as follows. For some finite type varieties $X$ over a field one can compute the Hodge-Deligne polynomial, which is given by the two-variable polynomial $\sum_k (-1)^k h^{p,q}(Gr_{W}^{p+q}H^*_c(X,\Q)) \, x^p y^q$. That is, it is the generating function with coefficient of $x^p y^q$ equal to the dimension of $p$-part of the Hodge filtration in the $(p+q)$th associated graded of the weight filtration. If $X$ were smooth and proper these are the classical Hodge numbers, but for varieties that are not smooth and proper one cannot determine the Betti numbers from the Hodge-Deligne polynomial without more information about the Hodge structure. In many cases there is however a simplest choice of Hodge structure compatible with the Hodge-Deligne polynomial and Occam's razor for Hodge structures obtains predictions by assuming that this simplest choice is the correct one. For certain families of locally closed subvarieties of symmetric products, Vakil and Wood are able to determine the limiting Hodge-Deligne polynomial using the theory of motivic zeta functions. Using this and Occam’s razor for Hodge structures, they are able to make conjectures about the limiting homology of these families of subvarieties.

We prove two of Vakil and Wood's conjectures, namely Conjecture G and Conjecture H. This can be seen as giving evidence supporting the predictive power of Occam's razor for Hodge structures. On the other hand, we also show that Formula 1.50 of \cite{VW} -- here equation \ref{eqnformula150} -- is incorrect, giving an example where Occam's razor for Hodge structures predicts the wrong answer. 

We now give a definition of the types of subspaces of symmetric products to which the conjectures pertain. Recall that the $k$th fold symmetric product of a space $X$, denoted $Sym^k (X)$, is the quotient of $X^k$ by the natural action of the symmetric group on $k$ letters. A point $\xi=(x_1, \ldots, x_k) \in Sym^k X$ determines a partition of the number $k$ by recording the multiplicity of each point $x_i$ appearing in $\xi$. For example, if $a$, $b$ and $c$ are distinct complex numbers, the configuration $(a,b,b,c) \in Sym^4(\C)$ corresponds to the partition $1+1+2$.

For $\lambda$ a partition of $k$, define $w_{\lambda}(X)$ to be the subspace of $Sym^k(X)$ of points whose corresponding partition is $\lambda$. Given a partition $\lambda=(a_1 + \ldots + a_n)$ of $k$, let $1^j \lambda$ be the partition $(1 + 1 + \ldots + 1 + a_1 + \ldots + a_n)$ of $j+k$ obtained by adding $j$ ones to $\lambda$. Note that the space $w_{1^j}(X)$ is the configuration space of $j$ unordered distinct particles in $M$. We can now state Conjectures G and H.

\begin{theorem}[Conjecture G]
\label{G}
For $j$ sufficiently large compared to $i$, we have that $H_i(w_{1^j}(\C P^2);\Q)=\Q$ for $i=0,2,4,7,9,11$ and $0$ otherwise. 
\end{theorem}

\begin{theorem}[Conjecture H]\label{thmconjectureh}
The colimit $\text{colim}_{j \to  \infty} H_i(w_{1^j2}(\C^d);\Q)$ is eventually periodic in $i$. 
\end{theorem}

Vakil and Wood also predicted in Formula 1.50 of \cite{VW} that the groups appearing in Theorem \ref{thmconjectureh} are as follows:
\begin{equation}\label{eqnformula150} \text{colim}_{j \to  \infty} H_i(w_{1^j2}(\C^d);\Q) = \begin{cases} \Q & \text{if $i = 0$} \\
\Q^2 & \text{if $i = 2(2k-1)d-1$ or $4kd$ for $k\geq 1$} \\
0 & \text{otherwise.}\end{cases}\end{equation}
This is not the case; they are instead equal to:
\begin{equation}\label{eqncorrects} \text{colim}_{j \to  \infty} H_i(w_{1^j2}(\C^d);\Q) = \begin{cases} \Q & \text{if $i = 0$} \\
\Q^2 & \text{if $i = k(2d-1)$ for $k\geq 1$} \\
0 & \text{otherwise.}\end{cases}\end{equation}
Thus the spaces $w_{1^j2}(\C^d)$ are an example where Occam's razor for Hodge structures yields incorrect guesses for stable homology groups.

The organization of this note is as follows. In Section 2, we recall scanning and homological stability results of McDuff in \cite{Mc1} and Segal in \cite{Se}. Using these theorems as well as some rational homology calculations from \cite{Ser} and \cite{MR}, we prove Conjecture G in Section 3 and Conjecture H in Section 4.

We have been informed that Orsola Tommasi has an alternative proof of Conjecture H. She observed that the cohomology groups of $w_{\lambda}(  \C^d )$ have pure weight. This allows one to determine the stable rational cohomology of $w_{1^j2}(\C^d)$ from the corresponding element in the Grothendieck ring of varieties, which was computed in \cite{VW}. The case of $d=1$ of Conjecture H has also been proven by Thomas Church, Jordan Ellenberg and Benson Farb (Proposition 4.4 of \cite{CEF}) using the machinery of \'etale cohomology and $L$-functions. This is part of their general program relating homological stability results for configuration spaces to the combinatorics of polynomials over finite fields.

\textbf{Acknowledgments:} We would like to thank Melanie Wood for alerting us to these conjectures as well as thank her, Ravi Vakil and Thomas Church for feedback on drafts of this note.

\section{Scanning} Our primary tool for proving these conjectures is Theorem 1.1 of McDuff in \cite{Mc1} regarding the so-called scanning map. For $M$ a manifold, let $\dot TM$ denote the fiberwise one-point compactification of the tangent bundle of $M$. Let $\Gamma^c_k(\dot TM)$ denote the space of compactly supported sections of $\dot TM \m M$ of degree $k$ (see \cite{Mc1} or \cite{BMi} for the definition of degree) topologized with the compact-open topology. In \cite{Mc1} McDuff constructed a map $s:w_{1^j}(M) \m \Gamma^c_k(\dot TM)$ and proved the following theorem.

\begin{theorem}
The scanning map $s:w_{1^j}(M) \m \Gamma^c_j(\dot TM)$ induces an isomorphism on homology groups $H_i$ for $j \gg i$. If $M$ is the interior of a compact connected manifold with non-empty boundary,  then the scanning map induces an injection on all homology groups.
\label{scan}
\end{theorem}

The isomorphism portion of the above theorem is Theorem 1.1 of \cite{Mc1} and the injectivity statement is contained in the proof of Theorem 4.5 of \cite{Mc1}. One can find explicit ranges for this theorem. By Proposition A.1 of \cite{Se}, for homology with integral coefficients $j \geq i/2$ suffices. By Theorem B of \cite{RW}, for homology with rational coefficients $j \geq i$ suffices if the dimension of $M$ is at least $3$.

\section{Conjecture G}

In this section we prove Conjecture G. By Theorem \ref{scan} this amounts to computing the rational homology groups of a space of sections. In principle this can be done using the techniques introduced in \cite{Su3}. However, we will first compare this space of sections with the space of continuous degree $l$ maps  $Map_l(\C P^2,S^4)$. This space is topologized with the compact open topology and the degree of a map is defined via the induced map on top dimensional homology. Because $\C P^2$ is compact the space of compactly supported sections is equal to the space of all sections. We use the following proposition, which follows from Proposition 3.5 of \cite{BMi}.

\begin{proposition}Let $M$ be a compact orientable manifold of dimension $n$ and let $\chi(M)$ denote its Euler characteristic. For all integers $k \neq \chi(M)/2$ and integers $l \neq 0$, $\Gamma^c_k(\dot TM)$ is rationally homotopy equivalent to $Map_l(M,S^n)$.
\label{BM}
\end{proposition}

The above proposition follows from the fact that the classifying space of oriented rational $S^n$ bundles is $n$-connected. This implies that the classifying map of such a bundle over an $n$-dimensional space is  null-homotopic. In particular, the rationalization of $\dot TM$ is a trivial rational $S^n$ bundle. This trivialization allows one to compare the components of  $\Gamma^c(\dot TM)$ and $Map(M,S^n)$.

Since we are interested in the limit as $k$ tends to infinity, the condition that $k \neq \chi(M)/2$ does not concern us. The spaces $Map_l(\C P^2,S^4)$ are simpler than the spaces $\Gamma^c_k(\dot T \C P^2)$ and their rational homology groups appear in the literature. For example, in  Example 2.5 (2) of \cite{MR}, M{\o}ller and Raussen proved the following using minimal models.

\begin{proposition}
Let $l$ be a non-zero integer. The commutative differential graded algebra $C^*(Map_l(\C P^m,S^{2m});\Q)$ of rational cochains is quasi-isomorphic to the free graded commutative algebra on classes $b_2,b_4 \ldots b_{2m-2}$ and $v_{2m+1}, v_{2m+3}, \ldots,  v_{4m-1}$ (subscripts indicate homological degree). The differential is given by: 

 \[ \partial(b_{2i})=0 \text{ for all }i \] \[ \partial(v_{4m-2i-1})= \sum_{r+s=2m-i} b_{2r}b_{2s} -\sum_{r+s=m} b_{2r}b_{2s} b_{2m-2i}  \text{ for } 0<i<m\] \[
\partial(v_{4m-1})=\frac{1}{4} \left(\sum_{r+s=m} b_{2r}b_{2s} \right)^2. \]

\label{model}
\end{proposition}

Specializing to the case $m=2$, we get the following. For $l$ a non-zero integer, the commutative differential graded algebra $C^*(Map_l(\C P^2,S^4);\Q)$ of rational cochains is quasi-isomorphic to the free graded commutative algebra on classes $b_2$, $v_5$, $v_7$ with the following differential: $\partial(b_2)=0$, $\partial(v_5)=-b_2^3$ and $\partial(v_7)=b_2^4/4$. 

We now have all of the ingredients needed to prove Conjecture G.

\begin{proof}[Proof of Theorem \ref{G}]
By Theorem \ref{scan} and Proposition \ref{BM} the stable homology of $w_{1^j}(\C P^2)$ agrees with that of $Map_l(\C P^2,S^4)$ with $l$ any nonzero integer. By Proposition \ref{model} a basis of a commutative differential graded algebra quasi-isomorphic to $C^*(Map_l(\C P^2,S^4);\Q)$ for $l \neq 0$ is given by classes of the form: $$b_2^i,\, b_2^i v_5,\, b_2^i v_7,\, b_2^i v_5 v_7$$ with $i \geq 0$. Using the description of the differential from Proposition \ref{model}, we can compute its kernel and image. Namely, a basis of the kernel of the differential is given by: $$b_2^i,\,b_2^i(b_2 v_5+4v_7)$$ with $i \geq 0$ and a basis of the image is given by: $$b_2^{3+i},\, b_2^{3+i}(b_2 v_5+4v_7)$$ with $i \geq 0$. Thus a basis of $H^*(Map_l(\C P^2,S^4);\Q)$ is given by the classes: $$1,\,b_2,\,b_2^2,\,(b_2 v_5+4v_7),\,b_2(b_2 v_5+4v_7),\,b_2^2(b_2 v_5+4v_7)$$ and so the rational cohomology vanishes except in dimensions $0,2,4,7,9,11$ where it is one dimensional. By the universal coefficient theorem, the rational homology also vanishes except in dimensions $0,2,4,7,9,11$ where it is one dimensional. 
\end{proof}

This proof gives slightly more information. The ring structure on the cohomology $H^*(Map_l(\C P^2,S^4);\Q)$ is induced by the ring structure on the commutative differential graded algebra described in Proposition \ref{model}. Since this algebra has a ring structure equal to that of a free graded commutative algebra, we see that for non-zero $l$ there is an isomorphism of rings \[H^*(Map_l(\C P^2,S^4);\Q)=\Q[b_2]/(b_2^3) \otimes \Lambda(c_7).\]
Here subscripts indicate homological degree, $c_7=b_2v_5+4v_7$ and $\Lambda$ denotes the exterior algebra with rational coefficients. From this we deduce an isomorphism of rings \[H^*(w_{1^j}(\C P^2);\Q)=\Q[b_2]/(b_2^3) \otimes \Lambda(c_7)\] in the range $* \leq j$, because the scanning map is an isomorphism in that range.

\begin{remark} Using the calculation of  $C^*(Map_l(\C P^m,S^{2m});\Q)$ for other $m$ and arguments similar to those of Theorem \ref{G}, one can calculate the stable rational cohomology of the spaces $w_{1^j}(\C P^m)$ for any $m$. For example, we have: 
\[\underset{j \to  \infty}{\text{colim}}\, H_i(w_{1^j}(\C P^3);\Q) = \begin{cases}
0 & \text{ if $i=1,3,5,7$ or $9$} \\
\Q & \text{ if $i=0,2,10,11,12,14,16,18,20$ or $22$} \\
\Q^3 & \text{ if $i=15$, $17$ or $19$} \\
\Q^2 & \text{ otherwise.} 
\end{cases} \] 
Moreover, as a ring, the stable rational cohomology is the graded commutative algebra generated by classes \[b_2,\,b_4,\,c_{11} \text{ and } c_{13}\] with relations generated by \[b_4^2-2b_2^2b_4 ,\, b_2b_4^2 \text{ and } b_4c_{13}-2b_2b_4c_{11}.  \] In personal communication with Vakil and Wood, we have been informed that these are not the groups predicted by Occam's razor for Hodge structures.
\end{remark}

\section{Conjecture H}

In this section we prove Conjecture H. First we note that $w_{1^j2}(\C^d)$ is homotopic to $w_{1^j}(\C^d-pt)$, where $\C^d -pt$ denotes $\C^d$ with a single point removed.

\begin{lemma}
The spaces $w_{1^j2}(\C^d) $ and $w_{1^j}(\C^d-pt)$ are homotopy equivalent.
\end{lemma}
\begin{proof}
Note that the map $w_{1^j2}(\C^d) \m \C^d$ recording the location of the two points which coincide is a fiber bundle. Its fibers are homeomorphic to  $w_{1^j}(\C^d-pt)$. Since $\C^d$ is contractible, the claim follows. 
\end{proof}

We can now apply Theorem \ref{scan} to study the stable homology groups of $w_{1^j2}(\C^d) \simeq w_{1^j}(\C^d-pt)$. In particular, Conjecture H will follow immediately from the following proposition.

\begin{proposition} \label{sectionpdisk}
The homology groups $ H_i(\Gamma_l^c(\dot T(\C^d -pt));\Q)$ are eventually periodic in $i$. Moreover, they are two-dimensional when $i=k(2d-1)$ for $k$ a positive integer, one dimensional when $i=0$ and zero otherwise, proving equation \ref{eqncorrects} is correct.
\end{proposition}

\begin{proof} We will compute the ring $H^*(\Gamma_l^c(\dot T(\C^d -pt));\Q)$. Because the tangent bundle of $\C^d$ is trivial, the bundle $\dot T(\C^d -pt)$ is also trivial and hence $\Gamma_l^c(\dot T(\C^d -pt))$ is homeomorphic to a component of the space of compactly supported maps $Map_l^c(\C^d -pt,S^{2d})$. The space $Map^c(\C^d -pt,S^{2d})$ is homotopic to the space of pointed maps $Map_{*}(\overline{\C^d-pt},S^{2d})$, where $\overline{\C^d-pt}$ is the one-point compactification of $\C^d-pt$. This one-point compactification is homotopy equivalent to $S^{2d}/(0 \sim \infty)$, which in turn is homotopy equivalent to $S^{2d} \vee S^1$, and hence we conclude that $Map^c(\C^d -pt,S^{2d})$ is homotopy equivalent to $\Omega^{2d} S^{2d} \times \Omega S^{2d}$ where $\Omega$ denotes the based loop space. It is well known that $H^*(\Omega S^{2d};\Q)=\Q[a_{4d-2}] \otimes \Lambda(a_{2d-1})$ with the subscript indicating homological degree. By Serre's calculation of the rational homotopy groups of spheres (page 500 of \cite{Ser}), $\Omega^{2d} S^{2d}$ is rationally homotopy equivalent to $\Z \times S^{2d-1}$ and $H^*(S^{2d-1};\Q) =\Lambda(b_{2d-1})$. By the K\"unneth theorem, we get $H^*(\Gamma_l^c(\dot T(\C^d -pt));\Q)= \Q[a_{4d-2}] \otimes \Lambda(a_{2d-1}) \otimes \Lambda(b_{2d-1})$. The claim now follows by the universal coefficient theorem. 
\end{proof}


As in the remarks following the proof of Theorem \ref{G}, the proof of this proposition allows one to compute the ring structure of $H^*(w_{1^j2}(\C^d);\Q)$ in the stable range. Note that while the above proposition establishes Conjecture H, it also contradicts equation \ref{eqnformula150}, giving a counterexample to Occam's razor for Hodge structures.

The eventual periodicity in the stable rational cohomology groups of the spaces $w_{1^j2}(\C^d)$ comes from the periodicity of $H^*(\Omega S^{2d};\Q)$ and the finite dimensionality of $H^*(\Omega^{2d} S^{2d};\Q)$. In general, one expects polynomial growth. 

\begin{remark}
In fact there is a straightforward procedure for computing the rational homology of $w_{\lambda}(\C^{d})$ for any partition $\lambda$. Let $G_n$ be the set of partitions of integers into only the numbers $1,2, \ldots, n$ and let $W_n$ denote the space $\bigsqcup_{\lambda \in G_n} w_{\lambda}(\C^d)$. Since $W_n$ is homotopic to the free algebra over the little $(2d)$-disks operad on the set $\{1, \ldots, n \}$, $H_*(W_n;\Q)$ is the free $(2d)$-Gerstenhaber algebra on $n$ classes in homological degree $0$. This follows from the mod $p$ calculations in Theorem 3.1 of part III of \cite{CLM} and knowledge of the Bockstein. Call these homology classes $x_1, \ldots, x_n$ and let $\phi : H_p(W_n;\Q) \otimes H_q(W_n;\Q) \m H_{p+q+2d-1}(W_n;\Q)$ denote the Gerstenhaber bracket. From this description, one can compute the rational homology of $w_{\lambda}(\C^{d})$.  In particular, $H_*(w_{1^j2}(\C^d);\Q)$ has a basis given by classes of the form: $$x_1^j x_2,\,x_1^{j-1}  \phi(x_1,x_2),\, x_1^{j-2}\phi(x_1,\phi(x_1,x_2)),\, x_1^{j-3} \phi(x_1,(\phi(x_1,\phi(x_1,x_2))), \dots$$ and $$x_1^{j-2}\phi(x_1,x_1)x_2,\, x_1^{j-3}\phi(x_1,x_1)\phi(x_1,x_2),\, x_1^{j-4}\phi(x_1,x_1)\phi(x_1,\phi(x_1,x_2)), \dots$$ where a class does not appear if the exponent on the $x_1$ term is negative. After counting how many elements appear in a given homological dimension, this description gives an alternative derivation of equation \ref{eqncorrects}.

One can also use this result to calculate the rational cohomology ring of $w_{1^j2}(\C^d)$ in a much bigger range. The previous calculation of the cohomology as a vector space tells us that $H^*(w_{1^j2}(\C^d);\Q)$ is abstractly isomorphic to  $H^*(\Gamma_l^c(\dot T(\C^d -pt);\Q)$ in the range $* \leq j(2d-1)-1$. Since the scanning map $s:w_{1^j2}(\C^d) \m \Gamma_l^c(\dot T(\C^d -pt))$ is injective on homology, it is surjective on cohomology. Thus, it induces an isomorphism in that range. Using the calculation of the ring structure of $H^*(\Gamma_l^c(\dot T(\C^d -pt));\Q)$ as in the proof of Proposition \ref{sectionpdisk} we conclude that as a ring the rational cohomology of $w_{1^j 2}(\C^d)$ is equal to $\Lambda(a_{2d-1})  \otimes  \Q[a_{4d-2}] \otimes \Lambda(b_{2d-1})$ in the range $* \leq j(2d-1)-1$.\end{remark}

\begin{remark}
One can also do calculations such as the following
\[\underset{j \to  \infty}{\text{colim}}\, H_i(w_{1^j2\,3}(\C^d);\Q) = \begin{cases}
\Q & \text{ if $i=0$} \\
\Q^{4k} & \text{ if $i=k(2d-1)$ for $k>1$} \\
0 & \text{ otherwise} 
\end{cases} \] 
either using the fiber sequence
\[w_{1^j}(\C^d - (pt \sqcup pt)) \to w_{1^j 2\,3}(\C^d) \to  w_{ 2\,3}(\C^d) \simeq S^{2d-1}\]
along the lines of Proposition \ref{sectionpdisk}, or by a free Gerstenhaber algebra calculation along the lines of the previous remark.
\end{remark}

\bibliography{theis4}{}
\bibliographystyle{alpha}

\end{document}